\documentclass{article}

\usepackage{amsmath, amsthm, amssymb}
\usepackage{amsfonts}
\usepackage{graphicx}
\usepackage{epic}

\newtheorem{Theorem}{Theorem}

\newtheorem{Example}{Example}

\begin{document}

%\begin{frontmatter}
%\title{$k$-Ribbon Fibonacci Tableaux}
%\author[Lewis]{Naiomi Cameron},
%\ead{ncameron@lclark.edu}
%\author[Pepper]{Kendra Killpatrick\corauthref{cor}}
%\corauth[cor]{Corresponding author.}
%\ead{Kendra.Killpatrick@pepperdine.edu}
%\address[Lewis]{Department of Mathematical Sciences, Lewis \& Clark College, Portland, Oregon 97219, USA}
%\address[Pepper]{Department of Mathematics, Pepperdine University, Malibu, California 90265, USA}  

\title{A Recursion for the FiboNarayana and the Generalized Narayana Numbers}         % Enter your title between curly braces
\author{Kristina Garrett\\
St. Olaf College\\
Northfield, MN 55057, USA\\ garrettk@stolaf.edu
\and
Kendra Killpatrick\\
Pepperdine University\\
Malibu, CA  90263, USA\\ 
kendra.killpatrick@pepperdine.edu} 

     % Enter your name between curly braces

\date{\today}          % Enter your date or \today between curly braces
\maketitle

\begin{abstract}  The Lucas polynomials, $\{n\}$, are polynomials in $s$ and $t$ given by $\{ n \} = s \{ n-1 \} + t \{ n-2 \}$ for $n \geq 2$ with $ \{ 0 \} = 0$ and $\{ 1 \} = 1$.  The lucanomial coefficients, an analogue of the binomial coefficients, are given by
\[
\Bigl\{ \begin{array}{c} n\\k \end{array} \Bigr \} = \frac{ \{n\}! }{ \{k\}! \{n-k\}!}.
\]
When $s = t = 1$ then $\{ n \} = F_n$ and the lucanomial coefficient becomes the fibonomial coefficient
\[
\binom{n}{k}_F = \frac{F_n!}{F_k! F_{n-k}!}.
\]
The well-known Narayana numbers, $N_{n,k}$ satisfy the equation
\[
N_{n,k} = \frac{1}{n} \binom{n}{k} \binom{n}{k-1}.
\]
%and sum to the Catalan numbers
%\[
%C_n = \sum_{k=1}^n N_{n,k}.
%\]
In 2018, Bennett, Carrillo, Machacek and Sagan defined the generalized Narayana numbers and conjectured that these numbers are positive integers for $n \geq 1$.  In this paper we define the FiboNarayana number $N_{n,k,F}$ and give a new recurrence relation for both the FiboNarayana numbers and the generalized Narayana numbers, proving the conjecture that these are positive integers for $n \geq 1$.    
%\[
%N_{\{n\}, \{k\}} = \frac{1}{\{n\}} \Bigl\{ \begin{array}{c} n\\k \end{array}\Bigr\} \Bigl\{ \begin{array}{c} n\\k-1 \end{array} \Bigr\}
%\]

%\[
%\frac{1}{n} \Bigl\{ \begin{array}{c} n\\k \end{array} \Bigr\} \Bigl\{\begin{array}{c} n\\k-1 \end{array} \Bigr\} = \Bigl\{ \begin{array}{c} n-1\\k-1 \end{array} \Bigr\}^2 + t \Bigl\{ \begin{array}{c} n-1\\k \end{array} \Big\} \Bigl\{ \begin{array}{c} n-1\\k-2 \end{array} \Bigr\}.
%\]

.

\end{abstract}

\vspace{.2in}


\section{Introduction}       % Enter section title between curly braces

The well-known Fibonacci sequence is defined recursively by $F_n = F_{n-1} + F_{n-2}$ with initial conditions $F_0 = 0$ and $F_1 = 1$.  The $n$th Fibonacci number, $F_n$, counts the number of tilings of a strip of length $n-1$ with squares of length 1 and dominos of length 2.

The Lucas polynomials $\{ n \}$ are defined in variables $s$ and $t$ as $\{ 0 \} = 0$, $\{ 1 \} = 1$ and for $n \geq 2$ we have $\{ n \} = s \{ n-1 \} + t \{ n-2 \}$.  If $s$ and $t$ are set to be integers then the sequence of numbers is called a Lucas sequence.  When $s = t = 1$, the sequence is the Fibonacci sequence with $\{ n \} = F_n$.  The lucanomials, an analogue of the binomial coefficients, are then defined as 
\[
\Bigl\{ \begin{array}{c} n\\k \end{array} \Bigr\} = \frac{ \{n\} !}{ \{k\}! \{n-k\}!}
\]
where $\{ n \}! = \{n\} \{n-1\} \cdots \{2\} \{1\}$.  
When $s = t = 1$, $\bigl\{ \begin{array}{c} n\\k \end{array} \bigr\}$ is known as the fibonomial coefficient and written
\[
\binom{n}{k}_F = \frac{F_n !}{F_k! F_{n-k}!}
\]
where $F_n! = F_n F_{n-1} \cdots F_2 F_1$.

In 1985, Gessel and Viennot \cite{GeV} gave a combinatorial interpretation of the fibonomial coefficients in terms of nonintersecting lattice paths and in 2008, Benjamin and Plott \cite{BeP} gave a combinatorial interpretation in terms of tilings.  In 2010, Sagan and Savage \cite{SaS} gave two straightforward combinatorial interpretations for the lucanomials based on tilings of a partition $\lambda$ that fits inside a $k \times (n-k)$ rectangle and the tilings of the complement of that partition.

  The Catalan numbers are given by the explicit formula

\[
C_n = \frac{1}{n+1} \binom{2n}{n}.
\]

The FiboCatalan number $C_{n,F}$ is then defined as
\[
C_{n,F} = \frac{1}{F_{n+1}} \binom{2n}{n}_F.
\]

More generally, the generalized Catalan number given in terms of the lucanomials is
\[
C_{\{n\}} = \frac{1}{\{n+1\}} \Bigl\{ \begin{array}{c} 2n\\n \end{array}  
\Bigr\}.
\]

In 2018, Bennett et. al. \cite{Ben} gave a new combinatorial interpretation for the lucanomial coefficients $\Bigl\{ \begin{array}{c} n\\k \end{array} \Bigr\}$ and a combinatorial interpretation for the generalized Catalan numbers.

The well-known Narayana numbers are defined as
\[
N_{n,k} =  \frac{1}{n} \binom{n}{k} \binom{n}{k-1}
\]

and are known to sum to the Catalan numbers:
\[
C_n = \sum_{k=1}^n N_{n,k}.
\]

We define the FiboNarayana number as
\[
N_{n,k,F} = \frac{1}{F_n} \binom{n}{k}_F \binom{n}{k-1}_F
\]
and prove that these numbers are positive integers for $n \geq 1$.

Bennett et. al. defined the generalized Narayana numbers in terms of the lucanomial coefficients as
\[
N_{ \{n\, k\}} = \frac{1}{\{n\}} \Bigl\{ \begin{array}{c} n\\k \end{array} \Bigr\} \Bigl\{ \begin{array}{c} n\\k-1 \end{array} \Bigr\}
\]
and conjectured that these are positive integers for $n \geq 1$.  

In this paper, we prove that for $n \geq 2$ the FiboNarayana numbers satisfy the recurrence
\[
\frac{1}{F_n} \binom{n}{k}_F \binom{n}{k-1}_F = \binom{n-1}{k-1}_F^2 + \binom{n-1}{k}_F \binom{n-1}{k-2}_F.
\]
 
and that for $n \geq 2$ the generalized Narayana numbers satisfy the recurrence

\[
\frac{1}{n} \Bigl\{ \begin{array}{c} n\\k \end{array} \Bigr\} \Bigl\{\begin{array}{c} n\\k-1 \end{array} \Bigr\} = \Bigl\{ \begin{array}{c} n-1\\k-1 \end{array} \Bigr\}^2 + t \Bigl\{ \begin{array}{c} n-1\\k \end{array} \Big\} \Bigl\{ \begin{array}{c} n-1\\k-2 \end{array} \Bigr\}
\]
which answers the conjecture of Bennett et. al. and proves that both are positive integers for $n \geq 1$.  We will give both an algebraic proof and a combinatorial proof of these results.  The algebraic proofs appear first in section 2, while the combinatorial proof will depend on a combinatorial bijection that Killpatrick and Weaver \cite{KiW} gave for the Sagan and Savage tilings giving the lucanomials.  We give this bijection in section 3 and then the combinatorial proof of the recurrence relations in section 4.

\section{Background}

A partition of $n$ is a sequence of nonnegative integers $\lambda = (\lambda_1, \lambda_2, \dots, \lambda_k)$ with $\lambda_i \geq \lambda_{i+1}$ and $\sum_{i=1}^k \lambda_i = n$.  The Ferrers diagram of a partition $\lambda = (\lambda_1, \lambda_2, \dots, \lambda_k)$ is an upper left-justified array of rows with $\lambda_i$ squares in row $i$.  

In 2010, Sagan and Savage \cite{SaS} gave the following combinatorial interpretation of the lucanomial coefficients in terms of tilings of partitions that fit inside a certain rectangle.  We will make use of this interpretation to give the combinatorial proof that the FiboNarayana numbers and the generalized Narayana numbers are positive integers for $n \geq 1$. 

We say that a nonnegative integer partition $\lambda$ is contained in a $k \times (n-k)$ rectangle if $\lambda$ has $k$ parts and each part $\lambda_i \leq n-k$.  If $\lambda$ is contained in a $k \times (n-k)$ rectangle, then $\lambda$ determines a second partition $\lambda^*$ whose parts are given by the lengths of the columns in the complement of $\lambda$ within the $k \times (n-k)$ rectangle.  

A linear tiling of $\lambda$ is a tiling of each row $\lambda_i$ of $\lambda$ with length 1 squares and length 2 dominos.  Let $L_{\lambda}$ denote the set of all linear tilings of $\lambda$ and let $L_{\lambda}'$ denote the set of all tilings of $\lambda'$ in which all rows of length greater than 0 begin with a domino.  The weight of any tiling $T$ in $L_{\lambda} \times L_{\lambda^*}'$ is given by
\[
w(T) =  s^{(\text{number of squares in $L_{\lambda} \times L_{\lambda^*}'$})} t^{(\text{number of dominos in $L_{\lambda} \times L_{\lambda^*}'$})}.
\]

For example, the following is a partition $\lambda$ that is contained in a $6 \times 5$ rectangle and the corresponding tiling of $\lambda$ and $\lambda^*$ which has weight $s^{12} t^{9}$.

\begin{center}
\setlength{\unitlength}{1cm}
\begin{center}
\begin{picture}(15,3.5)(1,1)

\put(6,1){\line(0,1){3}}
\put(6,4){\line(1,0){2.5}}
\put(6,1){\line(1,0){2.5}}

\put(6,1.5){\line(1,0){2.5}}
\put(6,2){\line(1,0){2.5}}
\put(6,2.5){\line(1,0){2.5}}
\put(6,3){\line(1,0){2.5}}
\put(6,3.5){\line(1,0){2.5}}

\put(6.5,1){\line(0,1){3}}
\put(7,1){\line(0,1){3}}
\put(7.5,1){\line(0,1){3}}
\put(8,1){\line(0,1){3}}
\put(8.5,1){\line(0,1){3}}

\thicklines
\put(6,1){\line(1,0){1}}
\put(7,1){\line(0,1){1}}
\put(7,2){\line(1,0){.5}}
\put(7.5,2){\line(0,1){1}}
\put(7.5,3){\line(1,0){.5}}
\put(8,3){\line(0,1){1}}
\put(8,4){\line(1,0){.5}}
\thinlines

\put(6.25,1.25){\line(1,0){.5}}
\put(6.15, 1.65){$\bullet$}
\put(6.15, 2.15){$\bullet$}
\put(6.25, 2.75){\line(1,0){.5}}
\put(6.15, 3.15){$\bullet$}
\put(6.15, 3.65){$\bullet$}
\put(6.65, 1.65){$\bullet$}
\put(6.75, 2.25){\line(1,0){.5}}
\put(6.65, 3.15){$\bullet$}
\put(6.75, 3.75){\line(1,0){.5}}
\put(7.15, 2.65){$\bullet$}
\put(7.25, 3.25){\line(1,0){.5}}
\put(7.65, 3.65){$\bullet$}

\put(7.25, 1.25){\line(0,1){.5}}
\put(7.75, 1.25){\line(0,1){.5}}
\put(7.65, 2.15){$\bullet$}
\put(7.65, 2.65){$\bullet$}
\put(8.25, 1.25){\line(0,1){.5}}
\put(8.25, 2.25){\line(0,1){.5}}
\put(8.15, 3.15){$\bullet$}
\put(8.15, 3.65){$\bullet$}

\end{picture}
\end{center}
\end{center}

\begin{Theorem}(Sagan and Savage, 2010)  
\[
\Bigl\{ \begin{array}{c} n\\k \end{array} \Bigr\} = \sum_{\lambda \subseteq k \times (n-k)} \sum_{T \in L_{\lambda} \times L_{\lambda^*}'} w(T).
\]

\end{Theorem}

Let $SSP_{\binom{n}{k}}$ denote the set of Sagan and Savage tilings as described above that fit inside a $k$ x $(n-k)$ rectangle.  Then by setting $s=t=1$ in the theorem above, we have
\[
\vert SSP_{\binom{n}{k}} \vert = \binom{n}{k}_F = \frac{F_n!}{F_k! F_{(n-k)}!}.
\]

Since it is clear by definition that $\binom{n}{k}_F = \binom{n}{n-k}_F$, we have that 
\[
\vert SSP_{\binom{n}{k}} \vert = \vert SSP_{\binom{n}{n-k}} \vert,
\]
thus the number of Sagan and Savage tilings that fit inside a $k$ x $n-k$ rectangle is the same as the number of such tilings that fit inside an $n-k$ x $k$ rectangle.

Killpatrick and Weaver gave a bijective proof that
\[
\vert SSP_{\binom{n}{k}} \vert F_k! F_{(n-k)}! = F_n!
\]
and we will make use of this bijection in our proof that the FiboNarayana numbers and the generalized Narayana numbers are positive integers for $n \geq 1$.

\section{The Recurrences}

\begin{Theorem}
For $n \geq 2$, the FiboNarayana numbers satisfy the recurrence
\[
\frac{1}{F_n} \binom{n}{k}_F \binom{n}{k-1}_F = \binom{n-1}{k-1}_F^2 + \binom{n-1}{k}_F \binom{n-1}{k-2}_F.
\]
\end{Theorem}

\begin{proof}
We will use the fact that a tiling of a strip of length $n-1$ either has a domino in positions $k-1$  to $k$ or not.  If it does not, then the first $k-1$ squares can be tiled in $F_k$ ways and the remaining $n-k$ squares can be tiled in $F_{n-k+1}$ ways.  If it does have a domino in positions $k-1$ to $k$, then the first $k-2$ squares can be tiled in $F_{k-1}$ was and the remaining $n-k-1$ squares following the domino can be tiled in $F_{n-k}$ ways.  Thus $F_n = F_k F_{n-k+1} + F_{k-1} F_{n-k}$.

\begin{align*}
N_{n,k,F} &= \frac{1}{F_n} \binom{n}{k}_F \binom{n}{k-1}_F \\
&= \frac{1}{F_n} \frac{F_n!}{F_k! F_{n-k}!} \frac{F_n!}{F_{k-1}! F_{n-k+1}!}\\
&= \frac{F_{n-1}! F_n F_{n-1}!}{F_k! F_{n-k}! F_{k-1}! F_{n-k+1}!}\\
&= \frac{(F_k F_{n-k+1} + F_{k-1} F_{n-k})F_{n-1}! F_{n-1}!}{F_k! F_{n-k}! F_{k-1}! F_{n-k+1}!}\\
&= \frac{F_k F_{n-k+1} F_{n-1}! F_{n-1}!}{F_k! F_{n-k}! F_{k-1}! F_{n-k+1}!} + \frac{F_{k-1} F_{n-k} F_{n-1}! F_{n-1}!}{F_k! F_{n-k}! F_{k-1}! F_{n-k+1}!}\\
&= \frac{F_{n-1}!}{F_{k-1}! F_{n-k}!} \frac{F_{n-1}!}{F_{k-1}! F_{n-k}!} + \frac{F_{n-1}!}{F_k! F_{n-k-1}!} \frac{F_{n-1}!}{F_{k-2}! F_{n-k+1}!}\\
&= \binom{n-1}{k-1}_F^2 + \binom{n-1}{k}_F \binom{n-1}{k-2}_F
\end{align*}

\end{proof}

A similar recurrence holds for the generalized Narayana numbers.
\begin{Theorem}
For $n \geq 2$, the generalized Narayana numbers satisfy the recurrence
\[
\frac{1}{n} \Bigl\{ \begin{array}{c} n\\k \end{array} \Bigr\} \Bigl\{\begin{array}{c} n\\k-1 \end{array} \Bigr\} = \Bigl\{ \begin{array}{c} n-1\\k-1 \end{array} \Bigr\}^2 + t \Bigl\{ \begin{array}{c} n-1\\k \end{array} \Big\} \Bigl\{ \begin{array}{c} n-1\\k-2 \end{array} \Bigr\}.
\]
\end{Theorem}

\begin{proof}

As before, we will use the fact that a tiling of a strip of length $n-1$ either has a domino in positions $k-1$  to $k$ or not.  If it does not, then the first $k-1$ squares can be tiled in $\{ k \}$ ways and the remaining $n-k$ squares can be tiled in $\{n-k+1\}$ ways.  If it does have a domino in positions $k-1$ to $k$, then the first $k-2$ squares can be tiled in $\{k-1\}$ was and the remaining $n-k-1$ squares following the domino can be tiled in $\{n-k\}$ ways.  Thus $\{ n \} = \{k\} \{n-k+1\} + t\{k-1\} \{n-k\}$.

\begin{align*}
N_{\{n, k\}} &= \frac{1}{\{n\}} \Bigl\{ \begin{array}{c} n\\k \end{array} \Bigr\} \Bigl\{ \begin{array}{c} n\\k-1 \end{array} \Bigr\}\\
&= \frac{1}{\{n\}} \frac{\{n\}!}{\{k\}! \{n-k\}!} \frac{\{n\}!}{\{k-1\}! \{n-k+1\}!}\\
&= \frac{\{n-1\}! \{n\} \{n-1\}!}{\{k\}! \{n-k\}! \{k-1\}! \{n-k+1\}!}\\
&= \frac{(\{k\} \{n-k+1\} + t \{k-1\} \{n-k\})\{n-1\}! \{n-1\}!}{\{k\}! \{n-k\}! \{k-1\}! \{n-k+1\}!}\\
&= \frac{\{k\} \{n-k+1\} \{n-1\}! \{n-1\}!}{\{k\}! \{n-k\}! \{k-1\}! \{n-k+1\}!} + t \frac{\{k-1\} \{n-k\} \{n-1\}! \{n-1\}!}{\{k\}! \{n-k\}! \{k-1\}! \{n-k+1\}!}\\
&= \frac{\{n-1\}!}{\{k-1\}! \{n-k\}!} \frac{\{n-1\}!}{\{k-1\}! \{n-k\}!} + t \frac{\{n-1\}!}{\{k\}! \{n-k-1\}!} \frac{\{n-1\}!}{\{k-2\}! \{n-k+1\}!}\\
&= \Bigl\{ \begin{array}{c}  n-1\\k-1 \end{array} \Bigr\}^2 + t \Bigl\{ \begin{array}{c}  n-1\\k \end{array} \Bigr\}  \Bigl\{ \begin{array}{c} n-1\\k-2 \end{array} \Bigr\}
\end{align*}
\end{proof}

Since both the fibonomials and lucanomials are known to be positive integers for $n \geq 1$, the above theorems prove that the FiboNarayana numbers and the generalized Narayana numbers are both positive integers for $n \geq 1$.

\section{A Fibonomial Bijection}

We now describe the bijection of Killpatrick and Weaver \cite{KiW} between any tiling of the rows of an $(n-1)$ x $(n-1)$ stairstep shape and a set of three tilings:  a tiling of the rows of a $(k-1)$ x $(k-1)$ stairstep shape, a tiling of the rows of an $(n-k-1)$ x $(n-k-1)$ stairstep shape, and a Sagan and Savage tiling of an $(n-k)$ x $k$ rectangle.

To begin, we will say two elements in positions $i$ and $i+1$ of the same row are \emph{breakable} if they are not connected by a domino and \emph{unbreakable} if they are connected by a domino. In addition, let $A$ denote the tilings of the rows of an $(n-1)$ x $(n-1)$ stairstep board corresponding to $F_n!$ and $B$ be the set of tilings corresponding to $ \vert SSP_{\binom{n}{k}} \vert \cdot F_k!F_{n-k}! $, i.e. a tiling of the rows of a $(k-1)$ x $(k-1)$ stairstep shape, a tiling of the rows of an $(n-k-1)$ x $(n-k-1)$ stairstep shape, and a Sagan and Savage tiling of an $(n-k)$ x $k$ rectangle.

\begin{center} 
\includegraphics[height=2in, width=4in]{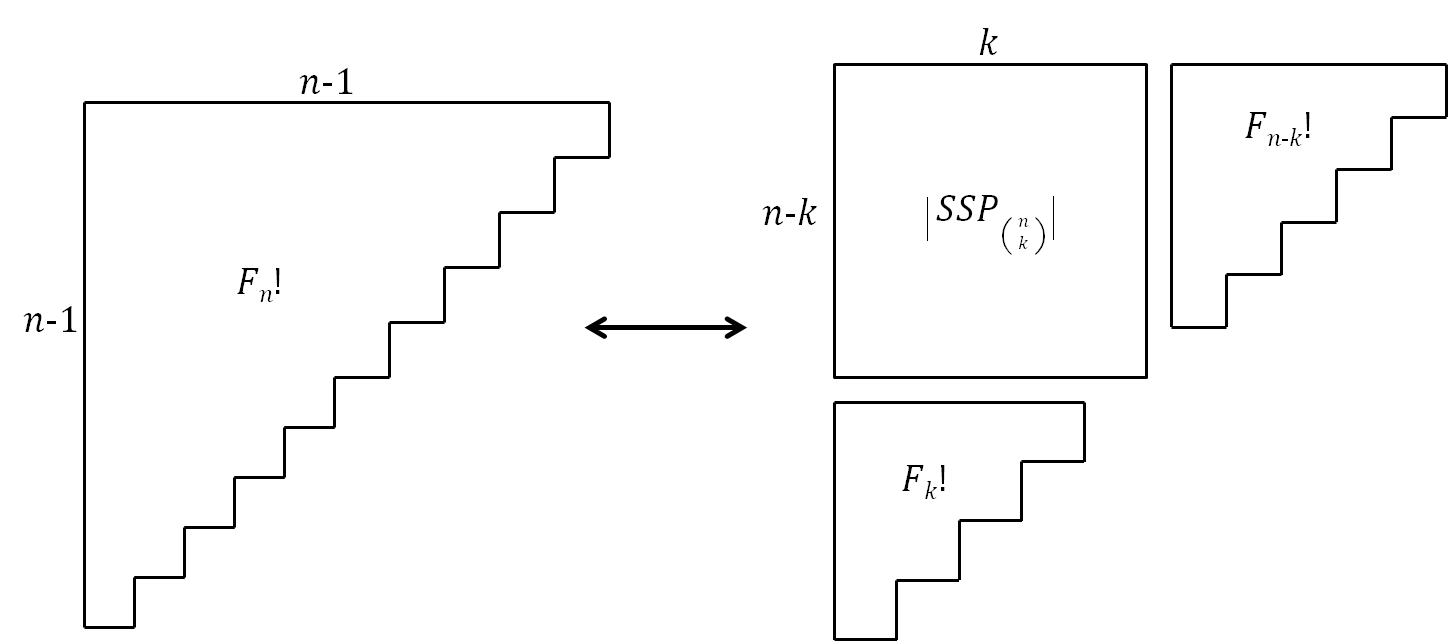} 
\end{center}

%\begin{figure}
%\begin{center}
%\includegraphics[height=2in, width=2in]{3Tilings.png}
%\end{center}
%\end{figure}

For the bijection, we will start with a tiling in $A$ and construct the three tilings in $B$.  Given a tiling in $A$, rows $n-k+1$ through $n-1$ give a tiling of a $(k-1)$ x $(k-1)$ stairstep shape so we will use this portion for the tiling in set $B$ counted by $F_{k}$.

Next we will construct both the path and associated tiling in $SSP_{\binom{n}{k}}$ and a tiling of the stairstep shape $(n-k-1)$ x $(n-k-1)$.  Start in the top row of $A$ and determine whether the elements in positions $k$ and $k+1$ are breakable. 
\begin{enumerate}
  \item If the elements are unbreakable, a leftward step is created in the upper right corner of the Sagan and Savage path in $B$; the segment consisting of the domino between elements $k$ and $k+1$ and all the elements in the top row of $A$ to the right of this domino are rotated 90 degrees counter-clockwise and placed under the leftward step of the path just created in $B$.  

\item If the elements in positions $k$ and $k+1$ are breakable, then create a downward step in the upper right corner of the Sagan and Savage path in $B$.  Place the elements in positions 1 through $k$ of the first row in $A$ to the left of the downward step in the path in $B$ just created, and the elements $k+1$ to $n-1$ in the same row of $B$, but to the right of the rectangle (this portion will be the first row of the tiling of the $(n-k-1)$ x $(n-k-1)$ stairstep shape).

\item If the elements in positions $k$ and $k+1$ in the first row of $A$ were breakable, consider the elements in positions $k$ and $k+1$ in the second row and repeat the process above.  If the elements in positions $k$ and $k+1$ in the first row of $A$ were unbreakable, examine elements $k-1$ and $k$ in the second row and repeat the process above. 
  
\item When reaching row $n-k+1$, if the elements to be compared are the elements in position $k$ and $k+1$, then there will be no element in position $k+1$ and the bijection is complete.  If the elements to be compared are elements $l$ and $l+1$ for $l < k$, then cycle around to the top row considering it as being "below" row $n-k$. Any rows or elements that have already been placed into the Sagan and Savage path tiling or the stairstep tiling are ignored.  Compare the elements in positions $l$ and $l+1$ in the first row that has elements that have not already been placed into one of the tilings and repeat the process above.  Any time row $n-k+1$ is reached, cycle back around to the top row of the tiling and continue the process.  

\item Continue until all of the elements of $A$ have been placed in $B$ and the corresponding stairstep tiling. The last step of the path made should then be connected to the bottom left corner of the rectangle with either leftward steps or downward steps, whichever is appropriate.
\end{enumerate}

\begin{Example}
Let $n=6$ and $k=3$.

%\begin{figure}
\begin{center}
\includegraphics[height=1in, width=2in]{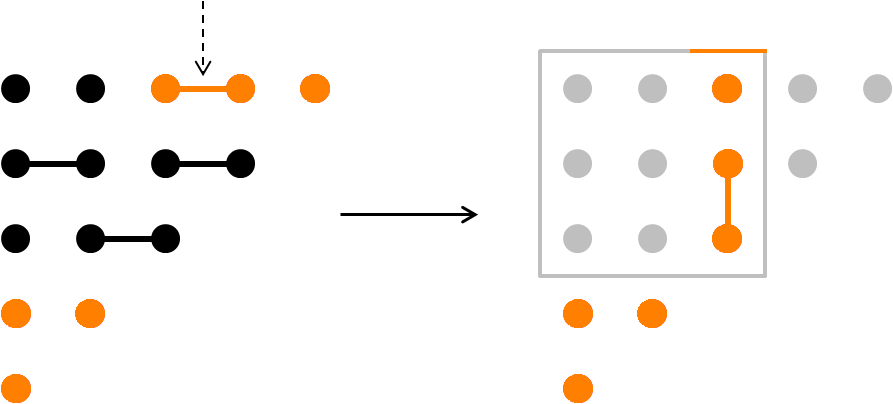}
\ \ \ \ \ \ \ \ \ \ 
\includegraphics[height=1in, width=2in]{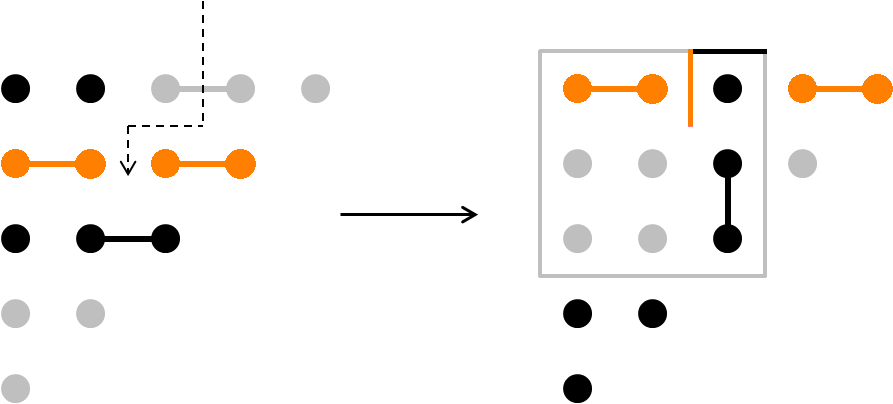}
\end{center}
%\end{figure}

%\begin{figure}
%\begin{center}
%\includegraphics[height=1in, width=2in]{BijectionStep2.png}
%\end{center}
%\end{figure}

%\begin{figure}
\begin{center}
\includegraphics[height=1in, width=2in]{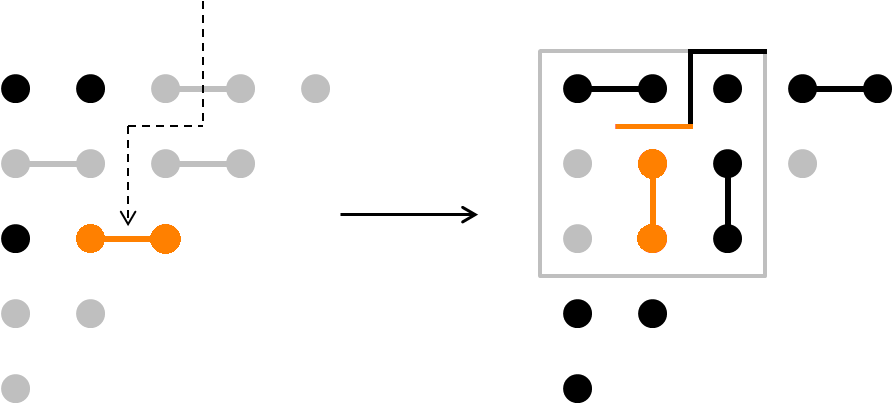}
\ \ \ \ \ \ \ \ \ \ 
\includegraphics[height=1in, width=2in]{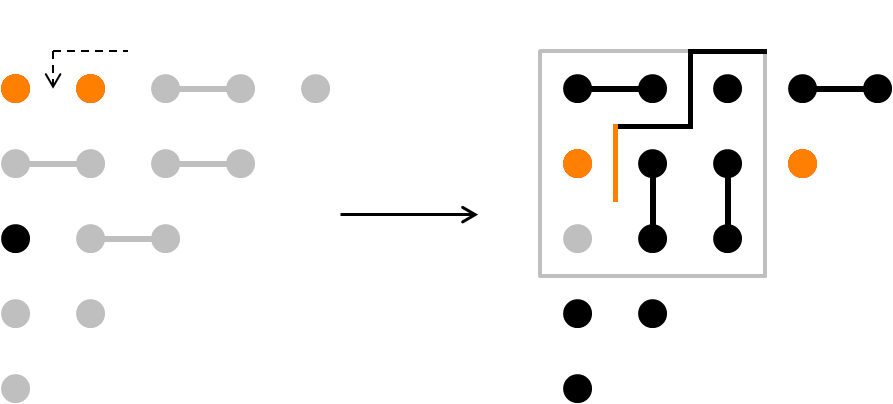}
\end{center}
%\end{figure}

%\begin{figure}
%\begin{center}
%\includegraphics[height=1in, width=2in]{BijectionStep4.png}
%\end{center}
%\end{figure}

%\begin{figure}
\begin{center}
\includegraphics[height=1in, width=2in]{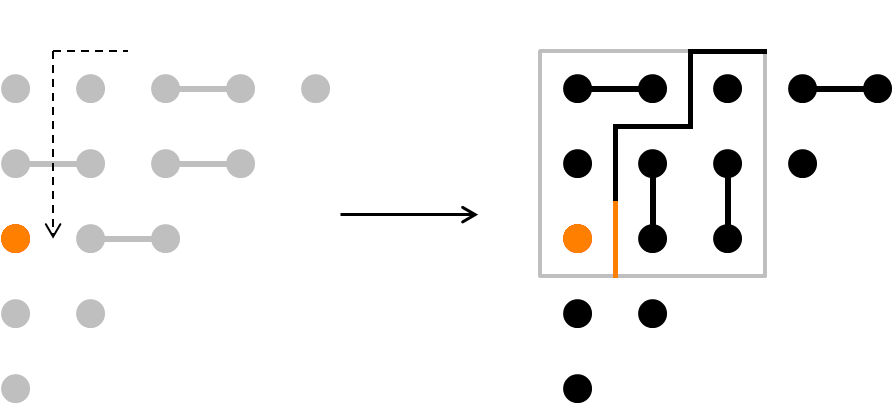}
\ \ \ \ \ \ \ \ \ \ 
\includegraphics[height=1in, width=2in]{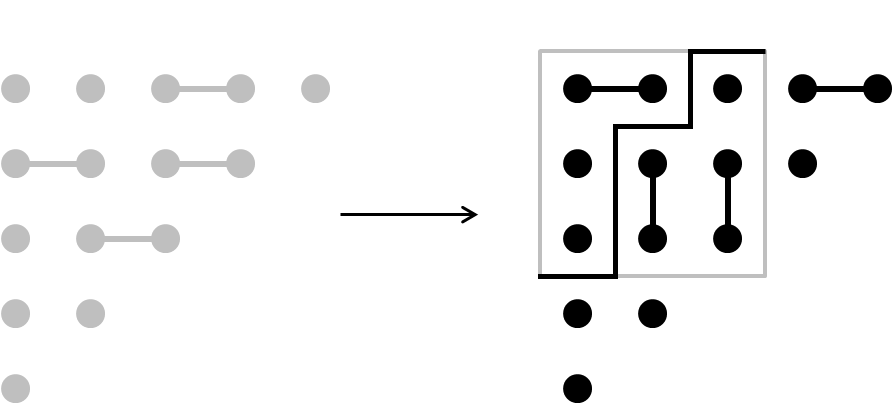}
\end{center}
%\end{figure}

%\begin{figure}
%\begin{center}
%\includegraphics[height=1in, width=2in]{BijectionFinal.png}
%\end{center}

\end{Example}

\section{A Combinatorial proof}

We will now give the combinatorial proof of the recurrence in Theorem 2 by showing why
\begin{align*}
F_n! F_{n-1}! &= F_k! F_{n-k}! F_{k-1}! F_{n-k+1}! \Bigl[\binom{n-1}{k-1}_F^2 + \binom{n-1}{k}_F \binom{n-1}{k-2}_F\Bigr]
\end{align*}

We will begin with a pair of tilings $(T_1, T_2)$ where $T_1$ is a tiling of a stairstep shape of size $n-1$ and $T_2$ is a tiling of a stairstep shape of size $n-2$, thus the set of all such possible pairs of tilings is counted by $F_n! F_{n-1}!$.

If there is no domino in positions $k-1$ to $k$ in $T_1$, then the number of possible tilings of the first row is $F_k F_{n-k}$.  Now use the Killpatrick and Weaver bijection to turn the tiling of the remaining stairstep shape of size $n-2$ into a Sagan and Savage tiling of size $n-k$ x $k-1$, a tiling of a stairstep shape of size $k-2$ and a tiling of a stairstep shape of size $n-k-1$.  The number of such triples is counted by $\binom{n-1}{k-1}_F F_{k-1}! F_{n-k}!$.  We will use the Killpatrick and Weaver bijection a second time to turn the tiling of $T_2$ into a Sagan and Savage tiling of size $n-k$ x $k-1$, a tiling of a stairstep shape of size $k-2$ and a tiling of a stairstep shape of size $n-k-1$.  The number of such triples is again counted by $\binom{n-1}{k-1}_F F_{k-1}! F_{n-k}!$.  Altogether we have 
\begin{align*}
F_k F_{n-k} \binom{n-1}{k-1}_F F_{k-1}! F_{n-k}! \binom{n-1}{k-1}_F F_{k-1}! F_{n-k}! \\
= F_k! F_{n-k}! F_{k-1}! F_{n-k+1}! \binom{n-1}{k-1}_F^2.
\end{align*}

\begin{center}
\includegraphics[height=2.5in, width=5in]{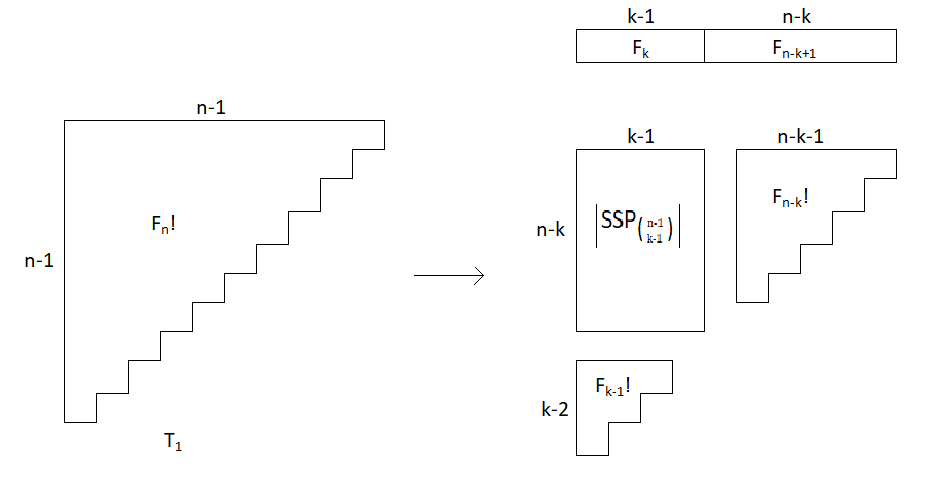}
\end{center}

\begin{center}
\includegraphics[height=2.5in, width=5in]{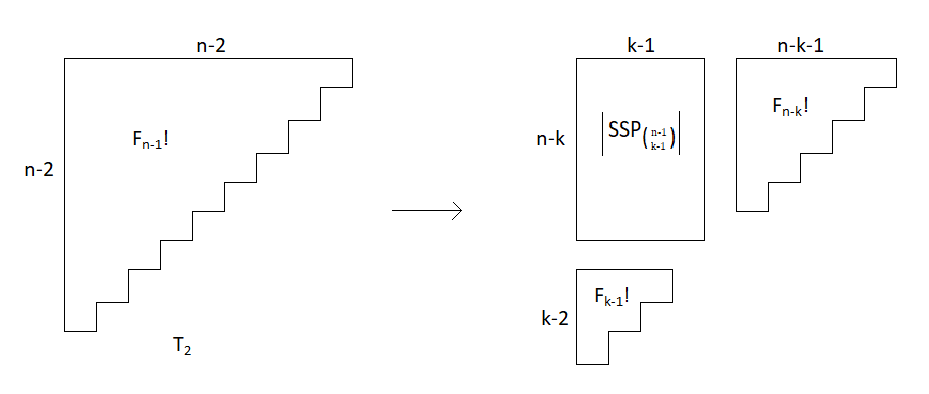} 
\end{center}

If there is a domino in positions $k-1$ to $k$, then the number of possible tilings of the first row is $F_{k-1} F_{n-k}$.  We will use the Killpatrick and Weaver bijection to turn the tiling of the remaining stairstep shape of size $n-2$ into a Sagan and Savage tiling of size $n-k+1$ x $k-2$, a tiling of a stairstep shape of size $k-3$ and a tiling of a stairstep shape of size $n-k$.  This triple is counted by $\binom{n-1}{k-2}_F F_{k-2}! F_{n-k+1}!$.  We will use the Killpatrick and Weaver a bijection a second time to turn the tiling of $T_2$ into a Sagan and Savage tiling of size $n-k-1$ x $k$, a tiling of a stairstep shape of size $k-1$ and a tiling of a stairstep shape of size $n-k-2$, which is counted by $\binom{n-1}{k}_F F_{k}! F_{n-k-1}!$.  Altogether we have 
\begin{align*}
F_{k-1} F_{n-k} \binom{n-1}{k-2}_F F_{k-2}! F_{n-k+1}! \binom{n-1}{k}_F F_{k}! F_{n-k-1}! \\ =
F_k! F_{n-k}! F_{k-1}! F_{n-k+1}! \binom{n-1}{k-2}_F \binom{n-1}{k}_F.
\end{align*}

\begin{center}
\includegraphics[height=2.5in, width=5in]{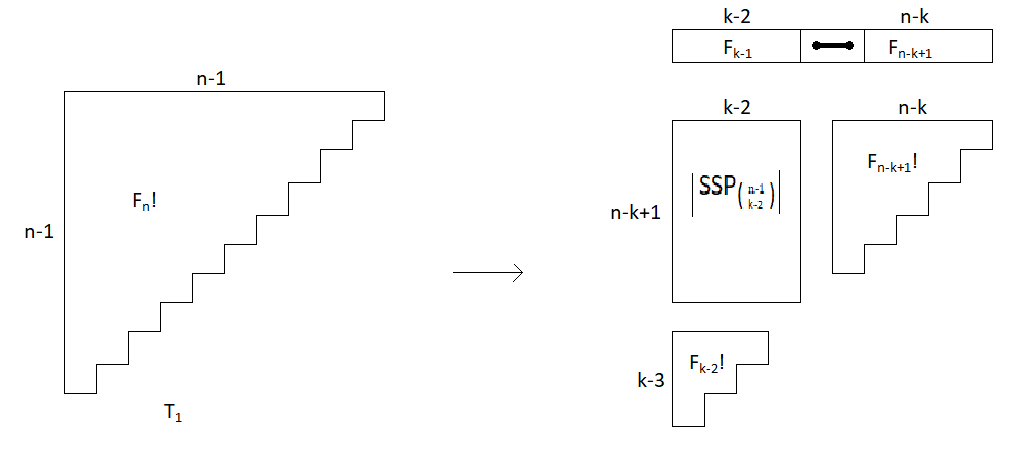}
\end{center}

\begin{center}
\includegraphics[height=2.5in, width=5in]{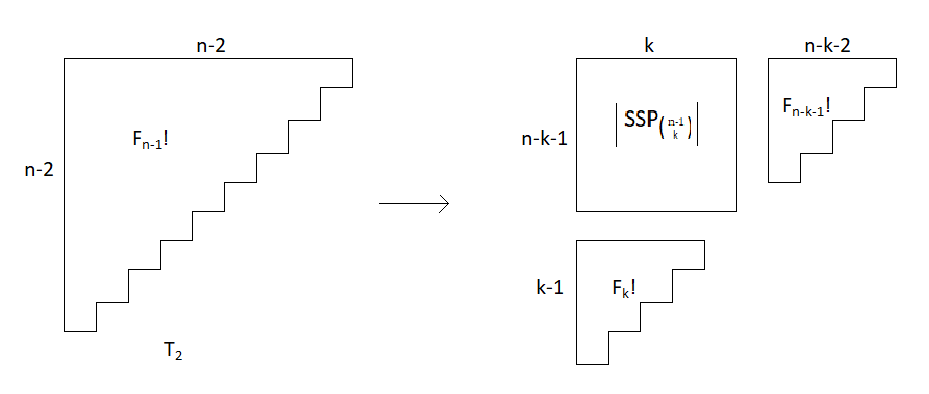} 
\end{center}

% Set the ending of a LaTeX document
\end{document}